\documentclass{amsart}

\usepackage{url}
\usepackage{color}
\definecolor{darkgreen}{rgb}{0,0.5,0}
\usepackage[
        colorlinks, citecolor=darkgreen,
        backref,
        pdfauthor={F. Bianchi, E. Kaya, J.S. M\"uller},
        pdftitle={Algorithms for $p$-adic Heights on Hyperelliptic Curves
        with Arbitrary Reduction}
]{hyperref}

\usepackage[utf8]{inputenc}
\usepackage[all]{xy} 		
\usepackage{tabularx}
\usepackage{enumitem}
\usepackage{colonequals,amsmath,amssymb,amsthm,mathtools}
\usepackage{times}
\usepackage[ruled,vlined]{algorithm2e}
\usepackage{float}
\usepackage{tikz,,tikz-cd,graphicx}
\graphicspath{ {./images/} }
\usepackage{hyperref}
\usepackage{comment}

\usepackage{thmtools}
\usepackage{thm-restate}
\theoremstyle{plain}
\newtheorem{thm}{Theorem}[section]

\newtheorem{prop}[thm]{Proposition}

\theoremstyle{definition}
\newtheorem{defn}[thm]{Definition}

\newtheorem{eg}[thm]{Example}

\theoremstyle{remark}
\newtheorem{rem}[thm]{Remark}

\def\presuper#1#2%
  {\mathop{}%
   \mathopen{\vphantom{#2}}^{#1}%
   \kern-12\scriptspace%
   #2}

\newcommand{\V}{{\operatorname{Vol}}}
\newcommand\Vint{\presuper\V\int}
\newcommand{\lV}{{\operatorname{Vol\ }}}
\newcommand\lVint{\presuper\lV\int}

\newcommand{\dR}{{\operatorname{dR}}}
\newcommand{\Supp}{{\operatorname{Supp}}}

\newcommand{\CG}{{\operatorname{CG}}}

\newcommand{\Res}{{\operatorname{Res}}}

\newcommand{\Div}{{\operatorname{Div}}}

\renewcommand{\div}{{\operatorname{div}}}

\newcommand*\angles[1]{\langle #1 \rangle}

\newcommand*\Z{\mathbb{Z}}
\newcommand*\Q{\mathbb{Q}}

\newcommand*\F{\mathbb{F}}
\renewcommand*\P{\mathbb{P}}

\newcommand*\Qp{\Q_p}

\renewcommand{\O}{\mathcal{O}}
\newcommand{\cO}{\mathcal{O}}
\newcommand{\B}{\mathcal{B}}

\newcommand*\cD{\mathcal{D}}

\newcommand*\cL{\mathcal{L}}

\numberwithin{equation}{section}

\definecolor{lightyellow}{RGB}{255, 255, 197}

\title[Algorithms for $p$-adic Heights]{Algorithms for $p$-adic Heights on Hyperelliptic Curves \\ of Arbitrary Reduction}

\subjclass[2020]{Primary: 14G40. Secondary: 11G50, 11S80, 14G05, 14F30.}

\author[Bianchi]{Francesca Bianchi}
\address{\hspace{-.2in} F. Bianchi}
\email{francesca.bianchi.maths@gmail.com}

\author[Kaya]{Enis Kaya}
\address{\hspace{-.2in} E. Kaya, Department of Mathematics, KU Leuven, Celestijnenlaan 200B, 3001 Heverlee, Belgium}
\email{enis.kaya@kuleuven.be}

\author[M\"uller]{J. Steffen M\"uller}
\address{\hspace{-.2in} J. S. M\"uller,  Bernoulli Institute, University of Groningen,  Nijenborgh 9,  9747 AG Groningen, The Netherlands}
\email{steffen.muller@rug.nl}

\date{\today}

\begin{document}

\maketitle

\begin{abstract}
In this paper, we develop an algorithm for computing Coleman--Gross (and hence Nekov\'a\v{r}) $p$-adic heights on hyperelliptic curves over number fields with arbitrary reduction type above $p$. This height is defined as a sum of local heights at each finite place and we use algorithms for Vologodsky integrals, developed by Katz and the second-named author, to compute the local heights above $p$. We also discuss an alternative method to compute these for odd degree genus~2 curves via $p$-adic sigma functions, via work of the first-named author. For both approaches one needs to choose a splitting of the Hodge filtration. A canonical choice for this is due to Blakestad in the case of an odd degree curve of genus~$2$ that has semistable ordinary reduction at $p$. We provide an algorithm to compute Blakestad's splitting, which is conjecturally the unit root splitting for the action of Frobenius. We give several numerical examples, including the first worked quadratic Chabauty example in the literature for a curve with bad reduction.
\end{abstract}

\setcounter{tocdepth}{1}
\tableofcontents

\section{Introduction}

Fix once and for all a prime number $p$. In the literature, there are several
definitions of $p$-adic height pairings on abelian varieties over number
fields, for instance due to Mazur--Tate \cite{MT83}. These are bilinear maps $A(F)\times A(F) \to \Q_p$, 
where $A$ is an abelian variety over a number field $F$.
Most of these constructions are quite similar to constructions of the
real-valued N\'eron--Tate height pairing.
Algorithms for computing $p$-adic heights 
\begin{itemize}
    \item allow one to compute $p$-adic regulators, some of which fit into
      $p$-adic versions of the Birch and Swinnerton--Dyer conjecture (\cite{MTT86,MST06,stein2013algorithms,BMS16}), and
    \item play a crucial role in carrying out the quadratic Chabauty method
      for integral (see \cite{balakrishnan2016quadratic}) and rational
      (see~\cite{BD18}) points on curves.
\end{itemize}
The $p$-adic height pairing $h$ constructed by Coleman--Gross \cite{CG89},
Colmez~\cite{colmez1998integration} and Besser \cite{BesserPairing} for the
case of a 
Jacobian variety $J$ of a curve $C/F$ has the advantage that it can be described
solely in terms of $C$. In order to define this pairing (see \S\ref{sec:CG_height}),
we first need to choose, as for most constructions of $p$-adic
heights,
\begin{itemize}
    \item a continuous id\`ele class character $\chi = (\chi_v)_{v}\colon \mathbb{A}_F^\times/F^\times \to \Qp$, and
    \item for each $v$ such that $\chi_v$ is ramified, a subspace $W_v\subset H_{\dR}^1(C_v/F_v)$ that is complementary to the subspace of holomorphic forms.
\end{itemize}
Here $F_v$ is the completion of $F$ at a finite place $v$ and $C_v$ is the
base change of $C$ to $F_v$. 
When the Jacobian variety $J_v$ of $C_v$ has semistable ordinary reduction, there is a natural choice for $W_v$, namely the unit root subspace for the action of the Frobenius endomorphism. 

The pairing $h$ is, by definition, a sum of
local height pairings $h_v$ between divisors on $C_v$ of degree~0 with
disjoint support, and the nature of $h_v$ depends on $\chi_v$:
\begin{itemize}
    \item When $\chi_v$ is unramified, $h_v$ is defined in terms of
      arithmetic intersection theory (see~\eqref{hvunram}) and its computation 
       is standard; see Remark~\ref{R:comp_away}. 
    \item When $\chi_v$ is ramified, on the other hand, the local pairing
      $h_v$ is defined using Vologodsky's theory of $p$-adic integration
      \cite{vologodsky2003hodge}. It can be computed as explained in
      \cite[\S5]{balakrishnan2012computing} and~\cite{GM23} if $C$ is a
      hyperelliptic curve of good reduction at $v$, where the Vologodsky
      integrals are in fact Coleman integrals. In
      Section~\ref{sec:Comp_CG_height}, we explain how to remove the good
      reduction assumption from this setting (see
      Algorithm~\ref{alg:CompCG}) using the algorithms in~\cite{Kaya,
      KatzKaya} developed by Katz and the second-named author.
Since by work of
Besser~\cite{Bes04, BesserPairing}, $h_v$ is equivalent to a
special case of the local $p$-adic height constructed by
    Nekov\'a\v{r}~\cite{Nek93} using $p$-adic Hodge theory, we also obtain an algorithm to
compute the latter on hyperelliptic curves.
\end{itemize}

In Section~\ref{sec:4}, we discuss the case of a genus~2 curve $C/F$. In
this case, the first-named author showed in~\cite{Bia23}
(using work of Colmez~\cite{colmez1998integration}) that $h_v$ can be
expressed in terms of $v$-adic sigma functions constructed
in~\cite[Section~3]{Bia23}, extending work of  Blakestad~\cite{blakestadsthesis}. 
This leads to an alternative algorithm to compute
$h_v$, implemented by the first-named author and available from~\cite{BiaCod}. 
In our earlier work~\cite{BKM23}, we extended some of the results in~\cite{Bia23} by expressing an extension of the Coleman--Gross height to certain divisors with \textit{common} support in terms of $v$-adic sigma functions. This then leads to an algorithm for the computation of such heights.

We also showed in~\cite{BKM23} that on general abelian varieties, $p$-adic N\'eron functions
give a local decomposition of a special case of the $p$-adic height
constructed by Mazur and Tate~\cite{MT83}, depending on the same data as
$h$. When $\chi_v$ is ramified and the abelian variety has semistable
ordinary reduction at $v$,
then Mazur and Tate constructed a \textit{canonical} local height pairing,
which is of particular importance; for instance, the $p$-adic regulator
in the $p$-adic Birch and Swinnerton-Dyer conjecture of
Mazur--Tate--Teitelbaum for elliptic curves with good ordinary and non-split
multiplicative reduction is defined in terms of it. 
When $J$ is the Jacobian of a genus $2$ curve $C/F$ 
given by a quintic
model that is semistable at $v$ such that $J$ has ordinary reduction at $v$, Blakestad
\cite{blakestadsthesis} constructed a complementary subspace 
$W^{C}_v$ (see~\S\ref{S:Blakestad}) such that the associated local Coleman--Gross height 
gives the canonical local Mazur--Tate height
(see~\cite[Theorem~5.23]{BKM23}).
As shown by the first-named author, $W^{C}_v$ is in fact the unit
root subspace when the reduction is good ordinary; see
\cite[Proposition~3.8]{Bia23}. We believe that this is still the case if
reduction is semistable ordinary; see \cite[Conjecture~3.17]{BKM23}. 
In \S\ref{sec:CompBlakSubs}, we explain how to compute $W^{C}_v$ numerically.

In Section~\ref{sec:Num_Exa}, we provide a number of numerical examples to
illustrate our algorithms. We show that various identities of
global heights predicted by identities in the Jacobian and 
comparison results between different constructions of local $p$-adic
heights hold to our working precision. Moreover, we describe the first
worked quadratic Chabauty example for a curve with bad reduction at $p$. Namely,
we show in Example~\ref{QCex} that the integral points on the genus~2 curve
$C\colon y^2 = x^5 + x^3 - 2x + 1$ are precisely $(0,\pm1), 
(\pm1, \pm1)$ by running quadratic Chabauty, as described
in~\S\ref{subsec:QC}, at the prime $p=5$ of bad reduction for $C$. This
example uses the techniques from~\cite{Bia23}.
Our examples are computed using \texttt{SageMath} \cite{SageMath} and \texttt{Magma} \cite{Magma}. Our code can be found at 
\[
\text{\url{https://github.com/KayaEnis/padicHeights}}.
\]

\subsection{Notation} Let $F$ denote a number field with ring of integers $\O_{F}$. For a finite place $v$ of $F$, we write $F_v$ and $\O_{v}$ for the completions at $v$. For each such $v$,
fix a uniformizer $\pi_v$ of $\O_v$ and let $\log_v\colon \O_v^{\times} \to F_v$ be the unique homomorphism extending
\[\log_v\colon 1 + \pi_v\O_v \to F_v,\ \ \ 1+z \mapsto \sum_{n=1}^{\infty} \frac{(-1)^{n+1}z^n}{n}.\]

Let $C$ be a smooth, projective and geometrically integral curve defined over $F$, and let $J/F$ be its Jacobian variety. By $\Div^0(C)$, we mean the group of degree $0$ divisors on $C$ defined over $F$.
For a finite place $v$ of $F$, we denote $C_v := C\otimes F_v$.

\subsection*{Acknowledgements} The authors would like to thank Jennifer
Balakrishnan, Amnon
Besser, Clifford Blakestad, Stevan
Gajovi\'c, Timo Keller, Oliver Lorscheid and Padmavathi Srinivasan 
for valuable discussions and comments. They also would like to thank the anonymous referees for several useful suggestions. F.B. and S.M. were supported by NWO grant VI.Vidi.192.106. E.K. was supported by NWO grant 613.009.124 and by FWO grant GYN-D9843-G0B1721N during various stages of this article. 

Parts of this paper have already appeared in the PhD dissertation \cite{KayaPhD} written by the second-named author at the University of Groningen under the supervision of the third-named author.

\section{Coleman--Gross heights}
\label{sec:CG_height}

The construction of the (extended) Coleman--Gross height pairing \cite{CG89,BesserPairing} is described in detail in \cite[\S2]{BKM23}. In this section, we give a brief summary for the convenience of the reader.

The pairing, which we denote by $h$, is a function from $\Div^0(C) \times \Div^0(C)$ to $\Qp$. It depends on some choices. First, let $\chi = (\chi_v)_{v}\colon \mathbb{A}_F^\times/F^\times \to \Qp$
be a continuous id\`ele class character. If $\chi_v$ is unramified, then
by continuity, $\chi_v$ is determined by $\chi_v(\pi_v)$. On the other hand, if $\chi_v$
is ramified, then we have $v\mid p$ and we may decompose $\chi_v$ on $\cO_{v}^\times$ as 
\begin{equation}\label{BranchOfP-adicLogComingFromEllv}
\chi_v = t_v\circ \log_v\,,
\end{equation}
 where $t_v$ is a $\Qp$-linear map from $F_v$ to $\Qp$. 
 Then there is a unique branch  $\log_v\colon F_v^\times \to F_v$ such
 that~\eqref{BranchOfP-adicLogComingFromEllv} extends to $F_v^\times$.
Second, for each $v$ such that $\chi_v$ is ramified, let $W_v$ be a
subspace of $H_{\dR}^1(C_v/F_v)$ that is complementary to the space of
classes of holomorphic forms.
We remark that, when $C_v$ has semistable ordinary reduction in the sense
of \cite[\S1.1]{MT83}, there is a canonical choice of such a complementary space: the unit root subspace for the action of the Frobenius endomorphism.

We now describe the pairing $h = h_{\chi,\{W_v\}_v}$. For two elements $D_1$ and $D_2$ in $\Div^0(C)$, we have
\[h(D_1,D_2)= \sum_v h_v(D_1\otimes F_v,D_2\otimes F_v)\,,\]
where the sum is over  finite places of $F$. Here $h_v$ is the local
height pairing at $v$ defined below, where we assume that $D_1$ and $D_2$
have disjoint support.

\begin{rem}\label{R:common}
One can still define $h_v(D_1,D_2)$ for $D_1$ and $D_2$ with common support thanks to the work of Balakrishnan--Besser~\cite{balakrishnanbesser2015}. This involves choosing a tangent vector at each point in $\Supp(D_1)\cap\Supp(D_2)$.
Although the local terms depend on the tangent vectors, the global height
  pairing does not, provided that we make these choices consistently at all
  places. See \cite[Sections 2 and 3]{balakrishnanbesser2015} or \cite[\S2.2]{BKM23} for details. 
\end{rem}

When $\chi_v$ is unramified, the local term is described using arithmetic
intersection theory; more precisely, by the proof
of~\cite[Proposition~1.2]{CG89}, we have
\begin{equation}\label{hvunram}
h_v(D_1,D_2) = \chi_v(\pi_v)\cdot(\cD_1\cdot \cD_2),
\end{equation}
where $(\cD_1\cdot \cD_2)$ denotes the intersection multiplicity of certain
extensions $\cD_1$ and $\cD_2$ of $D_1$ and $D_2$, respectively, to a
proper regular model of $C$ over $\cO_v$. The pairing $h_v(D_1,D_2)$ is
continuous, bi-additive and symmetric. 

The local term at a place $v$ such that $\chi_v$ is ramified (which
implies that $v\mid p$) is given in terms of a Vologodsky integral $\lVint$ (see~\cite[Theorem~B]{vologodsky2003hodge}
and~\cite[Theorem~2.1]{Bes05}). We need some notation first. We say a
differential of the third kind, if it is regular except possibly for simple
poles with integer residues. The differentials of the third kind on $C_v$
defined over $F_v$ form a group $T(F_v)$.
The residue divisor homomorphism $\Res\colon T(F_v)\to \Div^0(C_v)$ is given by $\nu\mapsto \sum_P \Res_P \nu \cdot (P)$
where the sum is taken over closed points of $C_v$, and by \cite[Proposition~2.5]{CG89}, there is a canonical homomorphism $\Psi\colon T(F_v)\to H_{\dR}^1(C_v/F_v)$
which is the identity on holomorphic differentials. Then there exists a unique form $\omega_{D_1}$ in $T(F_v)$ satisfying $\Res(\omega_{D_1})=D_1$ and $\Psi(\omega_{D_1})\in W_v$,
and the local pairing is defined by
\[h_v(D_1,D_2) = t_v\left(\Vint_{D_2}\omega_{D_1}\right).\]
Here we use the branch $\log_v$ of the $p$-adic logarithm determined
by the extension of~\eqref{BranchOfP-adicLogComingFromEllv} to
$F_v^\times$. The pairing $h_v(D_1,D_2)$ is continuous and bi-additive, but not symmetric
in general. It is symmetric 
when the subspace $W_v$ is isotropic
with respect to the cup product pairing $\cup$ on $H_{\dR}^1(C_v/F_v)$;
see~\cite[Proposition~5.2]{CG89}.

We finish this section by noting that the global pairing $h$ on $C$ factors
through the Jacobian variety $J$, since by~\cite[Proposition~1.2,
Proposition~5.2]{CG89} we have $h_v(D, \div(f)) = \chi_v(f(D))$ 
for any $D\in \Div^0(C_v)$ and $f\in F_v(C_v)^\times$ without zeros or poles
along $D$. That is to say the pairing $h = h_{\chi,\{W_v\}_v}$ induces a bilinear pairing $h \colon J(F)\times J(F)\to \Q_p$.
This pairing is symmetric if all $W_v$ are isotropic with respect to the cup product.

\section{Computing Coleman--Gross heights}
\label{sec:Comp_CG_height}

We keep the notation introduced in \S\ref{sec:CG_height}. Additionally, for a differential form $\rho$ of the second kind on $C_v$, denote by $[\rho]$ its cohomology class in $H_{\dR}^1(C_v/F_v)$.

\subsection{Local and global symbols}

A crucial step in computing $h_v(D_1,D_2)$, for a place $v$ with $\chi_v$
ramified, is the construction of the form $\omega_{D_1}$. This requires the
explicit computation of the homomorphism $\Psi$, but its original
definition is not suitable for this. As in
\cite{balakrishnan2012computing}, we express this map in terms of
Besser's \textit{local} and \textit{global symbols} in order to make it more explicit. 

\begin{defn} 
\label{LocalGlobalSymbolsDef}
Let $\omega$ be a form of the third kind on $C_v$. Let $\rho$ be a form of
  the second kind on $C_v$ and fix a point $Z$ at which $\rho$ is regular.
  The \textbf{local symbol} at a point $A \in C_v$ is defined as 
  $$\langle\omega,\rho\rangle_A := -\Res_A\left(\omega\lVint\rho\right)\,,$$ 
where by $\lVint\rho$ we mean the formal integral of a local expansion of
  $\rho$ around $A$ whose constant term is the Vologodsky integral
  $\lVint_Z^A\rho$ (resp. $\lVint_Z^{A^\prime}\rho$ for a nearby point
  $A^\prime$) if $\rho$ is regular (resp. singular) at $A$\footnote{When
  $\Res_A(\omega) = 0$ (in particular, when $\omega$ is regular at $A$),
  the choice of constant of integration does not matter.}. Then the
  \textbf{global symbol} is defined as $$\langle\omega,\rho\rangle := \sum_A
  \langle\omega,\rho\rangle_A,$$
where $A$ runs over all points where either $\omega$ or $\rho$ has a singularity. 
\end{defn}

The following result is due to Besser and allows us to compute $\Psi$ by computing global symbols.

\begin{prop}
\emph{(\cite[Corollary~3.14]{Bes05})}
\label{Corollary314besser2005p}
Let $\omega$ and $\rho$ be as in Definition~\ref{LocalGlobalSymbolsDef}. Then the global symbol $\angles{\omega,\rho}$ is nothing but the cup product $\Psi(\omega)\cup[\rho]$.
\end{prop}

\subsection{Computing Coleman--Gross heights}

In this subsection, we present an algorithm for computing Coleman--Gross $p$-adic height pairings on hyperelliptic curves. 

From now on, we assume that $p>2$ and that our fixed curve $C$ is a genus-$g$ hyperelliptic curve over $F$ with affine model 
\[y^2 = b(x), \ \ \ b(x)\in \O_F[x] \text{ is monic of degree } d.\] 
Choose $(\chi,\{W_v\}_v)$ as in \S\ref{sec:CG_height}. 

\begin{rem}\label{R:comp_away}
In \cite[\S3.1]{BMS16}, it is already explained how to compute the local height pairing $h_v$ at places $v$ with $\chi_v$ unramified; see \cite{holmes2012computing,muller2014computing,van2020explicit} for a detailed account. 
\end{rem}

Therefore our main interest lies in the local height pairings at places $v$
of ramifications for $\chi$, where $p$-adic integration is used. Fix a
finite place $v$ of $F$ with the property that $\chi_v$ is ramified. Recall that this can only happen for primes $v$ above $p$ and that from $\chi_v$ we obtained a branch $\log_v$ of the $p$-adic logarithm and a linear map $t_v$. 

\smallskip

\begin{algorithm}[H]
\label{alg:CompCG}
\caption{\bf Computing local heights at ramified places}
\KwIn{
\begin{itemize}
  \item A finite place $v\mid p$ of $F$ such that $\chi_v$ is ramified.
    \item A complementary subspace $W_v$, a branch $\log_v$ of the
      logarithm and a linear map $t_v$.
    \item Divisors $D_1,D_2\in \Div^0(C_v)$ with disjoint support.
\end{itemize}
}
\KwOut{The local height pairing $h_v(D_1,D_2)$.}

\begin{enumerate}
    \item Choose a differential $\omega$ on $C_v$ of the third kind defined over $F_v$ with residue divisor $D_1$ (see \S\ref{formofthethirdkindwithgivenresiduedivisor}).
    \item Determine the holomorphic form $\eta$ such that $\Psi(\omega - \eta)$ lies in the complementary subspace $W_v$ (see \S\ref{fromDtoomegaD}); then $\omega_{D_1} = \omega - \eta$.
    \item Compute the Vologodsky integral $\lVint_{D_2} \omega_{D_1}$ as described in \cite{KatzKaya,Kaya} with respect to $\log_v$ (see \S\ref{s:ComputingVol}).
    \item Return $t_v\left(\lVint_{D_2}\omega_{D_1}\right)$.
\end{enumerate} 
\end{algorithm}

\smallskip

Below, $\infty^+$ and $\infty^-$ stand for $(1:1:0)$ and $(1:-1:0)$ when $d$ is even, and $\omega_i$ denotes the $1$-form $x^i\frac{dx}{2y}$.

\subsubsection{Constructing a form with given residue divisor} 
\label{formofthethirdkindwithgivenresiduedivisor}

Let $D$ be a divisor of degree $0$ on $C_v$ over $F_v$ and write $D = \sum_j \left((P_j) - (Q_j)\right)$
with points $P_j,Q_j$ on $C_v$. Then $\omega = \sum_j \nu_j$ is a form of the third kind such that $\Res(\omega) = D$, where $\nu_j$ is defined by
\[\nu_j := 
\begin{dcases}
\left(\dfrac{y+y(P_j)}{x-x(P_j)}-\dfrac{y+y(Q_j)}{x-x(Q_j)}\right)\dfrac{dx}{2y} &\text{if } P_j \text{ and } Q_j \text{ are finite}; \\
\omit\hfil$\dfrac{y+y(P_j)}{x-x(P_j)}\dfrac{dx}{2y}$\hfil &\text{if } d
  \text{ is odd},\  P_j \text{ is finite},\ Q_j =\infty;\\
\omit\hfil$2\omega_g$\hfil &\text{if } d \text{ is even},\  P_j=\infty^{-},\ Q_j=\infty^{+}; \\
\omit\hfil$\dfrac{y+y(P_j)}{x-x(P_j)}\dfrac{dx}{2y}-\omega_g$\hfil &\text{if } d \text{ is even},\  P_j \text{ is finite},\ Q_j=\infty^{-}; \\
\omit\hfil$\dfrac{y+y(P_j)}{x-x(P_j)}\dfrac{dx}{2y}+\omega_g$\hfil &\text{if } d \text{ is even},\  P_j \text{ is finite},\ Q_j=\infty^{+}.
\end{dcases}\]
Note that the points $P_j,Q_j$, and hence the form $\nu_j$, are not
necessarily defined over $F_v$, but since $D$ is defined over $F_v$, so is
$\omega$. 

\subsubsection{Computing $\Psi$}
\label{computingPsi}

For $j=0,1,\dots,2g-1$, define
\[\rho_j := 
\begin{dcases}
\omit\hfil$\omega_j$\hfil &\text{if } j=0,\dots,g-1; \\
\omit\hfil$\omega_j$\hfil &\text{if } j=g,\dots,2g-1\ \text{and } d \text{ is odd};\\
\omega_{j+1} + 2\Res_{\infty^{+}}(\omega_{j+1})\omega_g &\text{if } j=g,\dots,2g-1\ \text{and } d \text{ is even}.
\end{dcases}\]
By construction, $\rho_0,\ldots,\rho_{g-1}$ span the space of holomorphic
differentials and each $\rho_j$ is of the second kind,
so that the class $[\rho_j]$ is an element of $H_{\dR}^1(C_v/F_v)$.
Moreover, the elements $[\rho_0],\dots,[\rho_{2g-1}]\in H_{\dR}^1(C_v/F_v)$
are linearly independent, which implies that the set $\B =
\{[\rho_0],\dots,[\rho_{2g-1}]\}$ forms a basis for $H_{\dR}^1(C_v/F_v)$.

Now let $\omega$ be a form of the third kind defined over $F_v$. Then
$\Psi(\omega) = \sum_{i = 0}^{2g-1} c_i[\rho_i]$
for some constants $c_i$. By Proposition~\ref{Corollary314besser2005p},
\begin{equation}
\label{Psi'sCoefficients}
\angles{\omega,\rho_j} = \Psi(\omega)\cup[\rho_j] = \sum_{i=0}^{2g-1}c_i([\rho_i]\cup[\rho_j]),\ \ \ j=0,1,\dots,2g-1.
\end{equation}
Let $N$ denote the cup product matrix with respect to the basis $\B$. From \eqref{Psi'sCoefficients}, we get
\[\begin{pmatrix}
c_0 \\
\vdots \\
c_{2g-1}
\end{pmatrix} = -N^{-1}\begin{pmatrix}
\angles{\omega,\rho_0} \\
\vdots \\
\angles{\omega,\rho_{2g-1}}
\end{pmatrix}.\]
Therefore, in order to compute $\Psi(\omega)$, it is enough the compute the matrix $N$ and the global symbols $\angles{\omega,\rho_j}$. The former is an easy task. The latter can be done using the techniques in \cite{KatzKaya,Kaya}; the situation is even better if the residue divisor of $\omega$ contains only affine points:

\begin{prop}
If the residue divisor $D$ of $\omega $ does not contain the point(s) at infinity, we have
\[\angles{\omega,\rho_j} = 
\begin{dcases}
\omit\hfil$-\lVint_D \rho_j - \Res_{\infty}\left(\omega\lVint\rho_j\right)$\hfil &\text{if } d \text{ is odd}, \\
\omit\hfil$-\lVint_D \rho_j - \Res_{\infty^{+}}\left(\omega\lVint\rho_j\right) - \Res_{\infty^{-}}\left(\omega\lVint\rho_j\right)$\hfil  &\text{if } d \text{ is even}.
\end{dcases}\]
\end{prop}

\begin{proof}
This is a straightforward generalisation of \cite[Proposition~5.12]{balakrishnan2012computing} (or rather, its corrected version in the errata \cite{balakrishnan2012computingErrata}). The key observation is that the local symbol at a point $A$ in the support of $D$ equals
\[-(\text{the multiplicity of } D \text{ at } A) \cdot \Vint_Z^A\rho_j\]
since $\omega$ has a simple pole at $A$. Here $Z$ is a fixed point throughout the global symbol computation. Therefore, summing over all points gives $-\lVint_D \rho_j$.
\end{proof}

\subsubsection{From $D$ to $\omega_D$} 
\label{fromDtoomegaD}

Let $\omega$ be a form of the third kind defined over $F_v$ with residue
divisor $D$. Let $\eta_0,\dots,\eta_{g-1}$ be differentials of the second
kind whose classes generate $W_v$. Since $W_v$ is a complementary to the
space of classes of holomorphic forms, there are constants $d_i,e_i$ such that 
\begin{equation}
\label{constantsdiei}
\Psi(\omega) = d_0[\rho_0] + \dots + d_{g-1}[\rho_{g-1}] + e_0[\eta_0] + \dots + e_{g-1}[\eta_{g-1}]. 
\end{equation}
Set
\[\eta = d_0\rho_0 + \dots + d_{g-1}\rho_{g-1}\in H^0(C_v,\Omega_{C_v/F_v}^1).\]
Then $\Psi(\omega - \eta)$ lies in $W_v$.

Let us explain how to determine the constants $d_i$. First, notice that $\eta_j/\omega_0$ can be represented by a polynomial in $x$. Without loss of generality, we may assume that
\begin{equation}
\label{DegreeAssumption}
\deg\left(\frac{\eta_j}{\omega_0}\right) = 
\begin{dcases}
\omit\hfil$j+g$\hfil &\text{if } d \text{ is odd}, \\
\omit\hfil$j+g+1$\hfil  &\text{if } d \text{ is even}.
\end{dcases} 
\end{equation}
As explained in \S\ref{computingPsi}, we can explicitly compute the constants $c_i$ such that
\begin{equation}
\label{constantsci}
\Psi(\omega) = c_0[\rho_0]+\dots+c_{2g-1}[\rho_{2g-1}]. 
\end{equation}
Finally, the assumption \eqref{DegreeAssumption} allows us to determine the constants $d_i$ easily by comparing \eqref{constantsdiei} and \eqref{constantsci}.

\subsubsection{Computing $\lVint_{D_2} \omega_{D_1}$}
\label{s:ComputingVol} 

The methods described in \cite{KatzKaya,Kaya} reduce the computation of
Vologodsky integrals on hyperelliptic curves to the computation of Coleman
integrals on hyperelliptic curves of good reduction. The latter is feasible
thanks to the algorithms developed in
\cite{BBKExplicit,balakrishnan2012computing,balakrishnan2015coleman}. In
particular, one uses \cite[Algorithm~4.8]{balakrishnan2012computing} to
deal with differential forms of the third kind; however, this requires the
good reduction hyperelliptic curve in question to have odd degree. 
Recent work of Gajovi\'c \cite{GM23} with the third-named author 
introduced a similar, but simpler 
algorithm, based on~\cite{balakrishnan2015coleman}, that also works for even degree. It is restricted to \textit{antisymmetric} differentials but is significantly faster; we make use of this algorithm when possible.

\begin{rem}\label{R:common_compute}
  The computation of Coleman--Gross heights between divisors with common
  support is somewhat more difficult. 
  For instance, suppose that $C$ is an odd degree hyperelliptic curve. For
  an affine point, fix the 
  tangent vector dual to $\omega_0$ and fix the tangent vector dual to
  $\omega_{g-1}$ for $\infty\in C(F)$. 
  According to~\cite[Theorem~2.2]{balakrishnan2016quadratic}, 
  the local height function
  $P\mapsto h_v(P-\infty, P-\infty)$ with respect to this choice is a
  double Vologodsky integral from a tangential base point at $\infty$ to
  $P$. 
  The algorithms for the computation of these local heights discussed
  in~\cite[\S4]{balakrishnan2016quadratic}
  and~\cite[\S3.3]{balakrishnan2017computing} are crucial to carry out the
  quadratic Chabauty method for integral points on odd degree hyperelliptic
  curves introduced for good reduction in~\cite{balakrishnan2016quadratic}.
  No algorithms for the computation of double Vologodsky integrals have
  been developed yet, but it should be possible to develop such an algorithm by combining the techniques of \cite{KatzKaya,Kaya} with the work of Katz and Litt \cite{KL22}.
  In the next section, we discuss an alternative
  approach to the computation of these local heights via $p$-adic N\'eron
  functions when $g=2$. 
\end{rem}

\section{$p$-adic N\'eron functions and the canonical subspace in genus~2}
\label{sec:4}

In this section we briefly discuss the relation between Coleman--Gross
heights and the construction of Mazur--Tate, referring to~\cite{BKM23}
and~\cite{Bia23} for
details. This
yields a different approach to the computation of the Coleman--Gross
height when the genus is~2 and $C(F_v)$ contains a Weierstrass point for
all $v$
of ramification of $\chi$. One advantage of this approach is that it makes
it possible to compute local height pairings between certain 
divisors with common support which are crucial for the quadratic Chabauty method for
integral points in~\cite{balakrishnan2016quadratic}; see Remark~\ref{R:common_compute}. If, in addition, the reduction of $C$ at such $v$ is semistable ordinary in the
sense of~\cite[Definition~5.1]{BKM23}, then this comparison also makes it
possible to isolate a canonical subspace $W^C_v$.

\subsection{Mazur--Tate heights, $p$-adic N\'eron functions and
Coleman--Gross heights}\label{subsec:}
Mazur and Tate~\cite{MT83} construct a
 global $p$-adic height pairing on an abelian variety, also depending
on the choice of a
continuous id\`ele class character $\chi = (\chi_v)_{v}\colon
\mathbb{A}_F^\times/F^\times \to \Qp$.
The
global pairing is a sum of local pairings, so called $\chi_v$-splittings
(see~\cite[\S3.1]{BKM23}).
In~\cite[\S3]{BKM23}, we introduced a notion of $p$-adic
N\'eron functions and showed that they induce $\chi_v$-splitting. 
When $v\mid p$ and $\chi_v$ is
ramified, the $p$-adic N\'eron function depends on the
choice of a splitting of the Hodge
filtration of $H_{\dR}^1$. The global Mazur--Tate height is then the sum of
the $p$-adic N\'eron functions by~\cite[Proposition~3.19]{BKM23}.

Suppose from now on that the abelian variety is the
Jacobian $J$ of a curve $C/F$. In this case the $p$-adic N\'eron
function, and hence the corresponding Mazur--Tate height, depends
on the choice of a complementary subspace $W_v$ as above.
More precisely, for a divisor $D\in \Div(J)$ and a finite place $v$ of $F$, the
$p$-adic N\'eron function $\lambda_{D,v}\colon J(F_v)\setminus\Supp(D)\to
\Q_p$ is, up to a constant, the same as the classical real-valued N\'eron function if $\chi_v$
is unramified; see~\cite[\S3.1.1]{BKM23}. In the ramified case, we use
Besser's $p$-adic log functions~\cite{Bes05, BMS} on the line bundle
$\O(D)$ to define $\lambda_{D,v}$ in~\cite[\S3.1.2]{BKM23}.

From now on, suppose for simplicity that 
$C/F$ is an odd degree hyperelliptic curve~\footnote{The local results
stated below for a place $v$ continue to hold if we only assume that
$C(F_v)$ contains a Weierstrass point.}.
Let $\Theta\in \Div(J)$ be the theta divisor corresponding to the
Abel--Jacobi embedding $\iota\colon C\to J$ with respect to $\infty\in C(F)$ and, for each $v$,
let $\lambda_v\colonequals \lambda_{2\Theta,v}$. Then we showed
in~\cite[Proposition~3.19]{BKM23} that the global Mazur--Tate
height function (with respect to $\chi$ and $W_v$, for ramified $\chi_v$)
is simply the sum of $\lambda_v$ for all $v$ (with the same choices). 
Moreover, the local Coleman--Gross height pairing between two divisors on $C$ with
disjoint support can be expressed in terms of pullbacks of $\lambda_v$.
This can be deduced from~\cite[Remark~6.10,
Corollary~6.12]{BMS}. 

We now assume that $p\geq 5$ and  that $C/F$ has genus 2, given by an affine equation
  \begin{equation}\label{Ceqn}
C\colon y^2 = b(x) =  x^5 + b_1x^4+b_2x^3+b_3x^2+b_4x+b_5,\ \ \ b_1,\ldots,b_5\in
\O_F\,.  
  \end{equation}
  In this case, we can describe the $p$-adic N\'eron function $\lambda_v$
  with respect to $\chi$ and $W_v$ explicitly by work of
  Blakestad~\cite{blakestadsthesis} and the first-named author~\cite{Bia23} as we
  now recall, referring to~\cite[Sections~3 and 4]{Bia23} for details. 
   We identify $H_{\dR}^1(C_v/F_v)$  and $H_{\dR}^1(J_v/F_v)$ via the
   isomorphism induced by $\iota$.
We assume for the rest of this section that $v\mid p$ and that $\chi_v$ is ramified.
  First,~\cite[Proposition~3.5]{Bia23} shows that the isotropic 
  complementary subspaces $W_v$ are in bijection with the set of symmetric matrices
  $c=(c_{ij})\in F_v^{2\times 2}$ via 
  \begin{equation}\label{cWv}
    c\mapsto \langle \eta_1^{(c)}, \eta_2^{(c)}\rangle,
  \end{equation}
  where 
\begin{equation}\label{eq:etais}
\begin{aligned}
\eta_1^{(c)} = (-3x^3 - 2b_1x^2 -b_2x + c_{12} x +
  c_{11})\frac{dx}{2y},\quad
 \eta_2^{(c)}= (-x^2 + c_{22} x + c_{12})\frac{dx}{2y}.
 \end{aligned}
\end{equation}
The $v$-adic $\sigma$-function $\sigma^{(c)}_v$ is defined
in~\cite[\S3.2]{Bia23} as the unique odd
power series of the form $T_1(1+\O(T_1,T_2))$ in $F_v[[T_1,T_2]]$ that
solves
\begin{equation*}
D_iD_j(\log(\sigma_v^{(c)}(T))) =-X_{ij}(T) + c_{ij}, \qquad \text{for all
  } 1\leq i,j\leq 2\,,
\end{equation*}
where $D_i$ is the invariant derivation dual to $\omega_i$ and and the
$X_{ij}$ are certain functions on the symmetric square of $C$ by Grant~\cite[\S4.1,
\S4.2]{Grant1990}. Via Grant's
explicit formal group law $T=(T_1,T_2)$ (see~\cite{Grant1990}
and~\cite[Section~2]{Bia23}),
$\sigma_v^{(c)}$ induces a function on a subgroup $H_v\subset J(F_v)$ of
finite index (see~\cite[Proposition~3.4]{Bia23}); more precisely, $H_v$ is
a subgroup of the model-dependent kernel of reduction. 
By~\cite[Corollary~4.6]{BKM23}, the $p$-adic N\'eron function with respect
to the subspace $W_v$ is then given by
\begin{equation}\label{lambdavdef}
\lambda_v(P) = -\frac{2}{ m^2}\cdot
  \chi_{v}\left(\frac{\sigma^{c(W_v)}_v(T(mP))}{\phi_m(P)}\right)\,,\quad mP\in H_v
  \,,m>0
\end{equation}
for any $P\in
  J(F_v)\setminus( \Supp(\Theta)\cup J(F_v)_{\mathrm{tors}})$; 
here $\phi_m$ is the $m$-th division polynomial
constructed by Kanayama~\cite{kan05} and~\cite{uch11}.
Using the symmetric $p$-adic Green functions due to
Colmez~\cite{colmez1998integration}, which by~\cite[Theorem~5.30]{Bia23}
are  essentially
$p$-adic N\'eron functions, one obtains
the following formula: 
\begin{prop}
\emph{(\cite[Corollary~4.7]{BKM23})}\label{prop:comp_disjoint}
  For distinct $P_1,P_2,Q_1,Q_2\in C(F_v)\setminus \{\infty\}$,
we have 
\begin{equation*}
  h_v^{\CG}(P_1 - P_2,Q_1 - Q_2) = -\frac{1}{2}\sum_{1\leq i,j\leq 2}
  (-1)^{i+j}\lambda_{v}([Q_i - P_j])\,,
\end{equation*}
with respect to the same choices $\chi_v$ and $W_v$ on both sides.
\end{prop}
By~\cite[Corollary~5.32]{Bia23}, this extends to a formula for
Coleman--Gross height pairings between general divisors with disjoint support.
The main result of~\cite{BKM23} extends this as follows.
Define the Coleman--Gross local height function 
 $$\lambda_v^{\CG}\colon J(F_v)\setminus\Supp(\Theta)\to \Q_p\,;\quad
[Q_1-Q_2]
\mapsto h_v(Q_1-Q_2, Q_1-Q_2)$$ (see~\cite[Definition~4.1]{BKM23}), where 
we extend the Coleman--Gross height pairing by choosing the tangent
vector dual to $\omega_0$ for the affine points $Q_1,Q_2$
(see Remarks~\ref{R:common} and~\ref{R:common_compute}). 
\begin{thm}\emph{(\cite[Theorem~1.1]{BKM23})}\label{T:bkm}
  We have $\lambda_v=\lambda_v^{\CG}$ 
with respect to the same choices $\chi_v$ and $W_v$ on both sides.
\end{thm}
\begin{rem}\label{R:ell}
  The genus~1 analogue of Theorem~\ref{T:bkm}  is due to
  Balakrishnan and Besser~\cite[\S4]{balakrishnanbesser2015}. Their proof
  is given for good ordinary
  reduction, but it extends readily to the general case.
  It is a classical result going back to 
  Bernardi~\cite{Ber81} and Mazur--Tate~\cite{MT83} that global $p$-adic heights
  on elliptic curves can be expressed in terms of $p$-adic sigma functions.
See for instance~\cite{bianchi2019quadratic} and the
references therein for the relation of the latter to $p$-adic N\'eron functions.
\end{rem}

An algorithm to compute $\lambda_v$ was implemented by the first-named
author and is
available from~\cite{BiaCod}. 
By Proposition~\ref{prop:comp_disjoint} and Theorem~\ref{T:bkm}, this
yields an alternative method for computing local
Coleman--Gross heights at $v$.
The algorithm needs a suitable matrix $c\in F_v^{2\times 2}$ as input. For instance, one
can simply take $c=0$ as in~\cite[\S3.1]{Bia23};  a more canonical choice is described
in~\S\ref{S:Blakestad} below.
Furthermore, one may use the results of~\cite{Bia23} to compute integrals
of differentials of the first, second and third kind,
see~\cite[\S6.7]{Bia23}.

\begin{rem}\label{R:compunram}
 Both Proposition~\ref{prop:comp_disjoint} and Theorem~\ref{T:bkm} also 
  hold true for primes $v$ of $F$ such that $\chi_v$ is unramified; in this
  case, both sides only depend on the local character $\chi_v$, and not on
  the choice of a complementary subspace. See~\cite[Remark~4.6]{Bia23}
  and~\cite[Theorem~4.2]{BKM23}.
\end{rem}

\subsubsection{Application to quadratic Chabauty}\label{subsec:QC}
Suppose now that $C$ is a genus 2 curve given by a model~\eqref{Ceqn} over
$F=\Q$ with Jacobian $J/\Q$ of Mordell--Weil rank~2. Suppose also that the
linear extension $\cL = (\cL_1,\cL_2)\colon J(\Q_p)\otimes \Q_p\to \Q_p^{\oplus 2}$ of
the abelian logarithm restricts to an injective map on $J(\Q)\otimes \Q$.
We normalise the character $\chi$ to be the cyclotomic character satisfying
$\chi_p(p)=\log_p(p)=0$, and we fix some complementary isotropic subspace
$W_p$.
The quadratic Chabauty method for
integral points on hyperelliptic curves, introduced
in~\cite{balakrishnan2016quadratic}, uses properties of local and global
heights to produce a locally analytic $p$-adic function
\begin{equation}\label{}
  \rho\colon C(\Z_p)\longrightarrow \Q_p 
\end{equation}
such that $\rho(C(\Z))$ takes values in an effectively computable finite set
$\Gamma$. It was reinterpreted in the Chabauty--Kim
framework~\cite{kim09:unipotent_albanese} and extended to rational points on certain curves (for
instance those that satisfy, 
in addition, $\mathrm{rk}\mathrm{NS}(J)>1$) by Balakrishnan and Dogra
in~\cite{BD18}.
So far, all examples computed using any instance of the quadratic Chabauty
method used primes $p$ of good reduction, although the theoretical setup
in~\cite{balakrishnan2016quadratic} requires no such
assumption.
In order to make the method from~\cite{balakrishnan2016quadratic} explicit,
one would have to compute with  local heights of the form $h_v(P-\infty,
P-\infty)$; 
see Remark~\ref{R:common_compute}. As discussed in~\cite[Section~7]{Bia23}, the comparison results
Proposition~\ref{prop:comp_disjoint} and Theorem~\ref{T:bkm} make it
possible to reformulate the quadratic Chabauty method in terms of $p$-adic
N\'eron functions, as we now explain. 

Since Theorem~\ref{T:bkm} only applies to
points outside the theta divisor, we cannot work directly with heights of
the form $h_v(P-\infty, P-\infty)$, so we use the simple identity 
\begin{equation}\label{hveasy}
  h_v(P-\infty, P-\infty) = \frac14\left(h_v(P-w(P),
 P-w(P))+\chi_v(4y(P))^2\right)
\end{equation}
proved (for $v\mid p$) in~\cite[Section~4]{balakrishnan2016quadratic}. We define
$\nu_v\colon C(\Q_v)\setminus\{P\,:\, y(P)\in \{0,\infty\}\}\to \Q_p$ by 
\begin{equation}\label{nuv}
  \nu_v(P)\colonequals \lambda_v(2\iota(P))+2\chi_v(2y(P))
\end{equation}
where $\lambda_p$ is the $p$-adic N\'eron function with respect to $\chi_p$
and $W_p$.
From Theorem~\ref{T:bkm}
and~\cite[Theorem~3.1]{balakrishnan2016quadratic}, we deduce:
\begin{thm}\label{T:qc}
  There are constants $a_1, a_2, a_3\in \Q_p$ and a finite
  set $\Gamma\subset \Q_p$, all effectively computable, such that
  the locally analytic function $\rho\colon C(\Z_p)\setminus\{P\,:\, y(P)=0\}\longrightarrow \Q_p $
  defined by
  \begin{equation}\label{rho}
    \rho(P)\colonequals \nu_p(P) -
    a_1\cL_1^2(P)-a_2\cL_1(P)\cL_2(P)-a_3\cL_2^2(P)
  \end{equation}
  takes values in $\Gamma$ for $P\in C(\Z)$.

  The elements of $\Gamma$ are of the form $\sum_v \gamma_v$, where $v$
  runs through the bad primes of $C$ and $$\gamma_v\in \Gamma_v\colonequals
  \{-\nu_v(P)\,:\, P\in C(\Z_v)\,,\; y(P)\ne 0\}\,.$$
\end{thm}
To find the constants $a_i$ in practice, 
one uses the fact that the global height function $P\mapsto h(P)
\colonequals h(P,P)$
is a $\Q_p$-valued quadratic form on $J(\Q)\otimes \Q$ 
and that the functions $\cL_1^2, \cL_1\cL_2,\cL_2^2$
span the space of such functions.
The sets $\Gamma_v$ may be computed either using the
techniques in~\cite[\S3.4]{balakrishnan2017computing} or via the proof
of~\cite[Lemma~7.2]{Bia23}. 
Finally, we may solve for the set of all $z\in C(\Z_p)$ such that
$\rho(z)\in \Gamma$. Unless $\rho$ is constant, this will result in a finite superset of $C(\Z)$. From this, we
can then try to provably compute $C(\Z)$ using the Mordell--Weil
sieve~\cite{bruin_stoll_2010}; see, for
instance, \cite{balakrishnan2017computing}.

\begin{rem}\label{R:}
In fact, the first-named author recovered in~\cite[Section~7]{Bia23} the main
results of~\cite{balakrishnan2016quadratic} in a simpler and more direct
  way, working entirely with $p$-adic N\'eron functions. This is then used
  in~loc. cit. to develop an explicit quadratic Chabauty method for rational points on
  certain bihyperelliptic genus~4 curves.
\end{rem}

\subsection{The canonical subspace}\label{S:Blakestad}
We now discuss briefly how to construct and compute a canonical complementary
subspace $W^{C}_v$ when the curve $C/F$ given by~\eqref{Ceqn} has semistable
ordinary reduction in the following sense. See~\cite[\S5]{BKM23} for
details. 

\begin{defn}\label{D:SemOrd}
  We say that an abelian variety $A/F$ has \textit{semistable ordinary reduction} at a
  finite prime $v$ if 
the connected component of the special fibre of the N\'eron
model of $A$ over $\O_v$ is an extension of an ordinary abelian variety by
  a torus. If $A=J$ is the Jacobian of a curve $C/F$ given by a
  model~\eqref{Ceqn}, then we say that $C$ has \textit{semistable ordinary
  reduction} at $v$ if 
  $J$ has ordinary reduction at $v$ and 
  if the Zariski closure of~\eqref{Ceqn}
  in $\P^{1,3,1}_{\O_v}$ is semistable. 
\end{defn}
When the Jacobian $J$ of $C$ has semistable ordinary reduction at $v$, Mazur--Tate construct a canonical $\chi_v$-splitting
using formal completions (see~\cite[\S1.9]{MT83}).

We now summarise the construction of an isotropic complementary subspace
$W^{C}_v$ that induces the canonical Mazur--Tate splitting, referring to~\cite[\S5.2]{BKM23}
for details. The construction is a straightforward generalisation of results of
Blakestad~\cite[Chapter~3]{blakestadsthesis} who assumed that
 $C$ has good reduction and $J$ has ordinary reduction.

Consider the local parameter $t = -\frac{x^2}{y}$ at $\infty\in
C(F)$. For every positive integer $n$, there exist
unique functions $\phi_n$ and $\psi_n$ on $C$, regular away from $\infty$, with $t$-expansions
\begin{align*}\label{BlakestadPhinPsin}
\begin{split}
\phi_n &= \frac{1}{t^{3p^n}} + \frac{A_n}{t^3} + \frac{B_n}{t} + I_nt + R_nt^3 + O(t^5), \ \ \ A_n,B_n,I_n,R_n\in\O_F;\\
\psi_n &= \frac{1}{t^{p^n}} + \frac{C_n}{t^3} + \frac{D_n}{t} + J_nt + S_nt^3 + O(t^5), \ \ \ C_n,D_n,J_n,S_n\in\O_F;
\end{split}
\end{align*}
see \cite[\S3.2.1]{blakestadsthesis}. 
By~\cite[Lemma~5.17, Proposition~5.19]{BKM23}, the matrices
$$M_n \coloneqq \begin{pmatrix}
A_n & B_n \\
C_n & D_n
\end{pmatrix}$$ are invertible over $\O_v$ for all $n$ precisely when $C$ has semistable
ordinary reduction at $v$.
If this is satisfied, then by~\cite[Lemma~5.17]{BKM23}, the limit
$$
\begin{pmatrix}
\alpha & \delta \\
\beta & \gamma
\end{pmatrix} \colonequals 
\lim_{n\to \infty}
\begin{pmatrix}
\alpha_n & \delta_n \\
\beta_n & \gamma_n
\end{pmatrix};
\qquad
\begin{pmatrix}
\alpha_n & \delta_n \\
\beta_n & \gamma_n
\end{pmatrix} \colonequals
M_n^{-1} \begin{pmatrix}
I_n & R_n \\
J_n & S_n 
\end{pmatrix} $$ 
exists and lies in $\O_v^{2\times 2}$.

\begin{prop}\emph{(\cite[Theorem~5.23]{BKM23})}\label{P:}
Suppose that $C$ has semistable ordinary reduction at $v$.
Then the subspace
$W_v^{C}$ corresponding under~\eqref{cWv} to 
\begin{equation}\label{eq:fijC}
\begin{aligned}
  c^{C}_{11} &\colonequals 2b_1b_2 - b_1\alpha + b_1^2\beta + 3\delta - 3b_1\gamma + 3b_3, \qquad
  && c^{C}_{12} \colonequals b_2 - b_1\beta + 3\gamma, \\
  c^{C}_{21} &\colonequals b_2 + \alpha - b_1\beta
  , \qquad && c^{C}_{22} \colonequals \beta
\end{aligned}
\end{equation}
  induces the canonical Mazur--Tate splitting.
\end{prop}
Hence  the
$p$-adic N\'eron function with respect to the matrix $c^{C}$ gives the
local contribution at $v$ of the canonical Mazur--Tate height function.

\subsubsection{Computing the canonical subspace}
\label{sec:CompBlakSubs}

In order to compute
the canonical $p$-adic N\'eron function or the corresponding
Coleman--Gross height, we need to compute the subspace $W^{C}_v$.
When $C_v$ has good reduction and $J_v$  has 
ordinary reduction, it was shown by the first-named author that
$W^{C}_v$ is in fact the unit root subspace for the action of the Frobenius 
(see \cite[Proposition~3.8]{Bia23}) so that  
one can use \cite[Proposition~6.1]{balakrishnan2012computing} to compute a
basis for $W^{C}_v$ in a simple and efficient way. 
In \cite{BKM23}, we conjectured that $W^{C}_v$ is still the unit root subspace
if the reduction is semistable ordinary; see Conjecture~3.17 of loc. cit.
If this is true, then it should be possible to compute it using
the description of the action of Frobenius on $H_{\dR}^1$ due to Coleman--Iovita
\cite{coleman_iovita:frobenius_monodromy}.

In any case, we can approximate  
$W^{C}_v$ to $n$ digits of precision by computing the functions
$\phi_n$ and $\psi_n$ up to $O(t^5)$.
Blakestad's~\cite[Proposition~22]{blakestadsthesis} (slightly adapted to 
arbitrary reduction in~\cite[Lemma~5.17]{BKM23}) states that for every $k\ne 1,3$, there is a unique
function $\rho_k\in F(C)$ such that $\rho_k$ has no poles outside $\infty$
and such that the expansion of $\rho_k$ around $\infty$ is of the form
\begin{equation}\label{rhokexpansion}
  \rho_k(t) = t^{-k}+M_kt^{-3}+N_kt^{-1}+\O(t)\,.
\end{equation}
In other words, we need to compute $\rho_{3p^n}$ and $\rho_{p^n}$ to
precision $\O(t^5)$. To accomplish this, 
we simply follow the proof of~\cite[Proposition~22]{blakestadsthesis},
which is based on Riemann--Roch.
For each $m\ne 1,3$, write $m=2i+5j$ and set $\tilde{\rho}_m=x^i(-y)^j$.
Then $\tilde{\rho}_m \in \cL(m\infty)\setminus\cL((m-1)\infty)$ and
$\tilde{\rho}_m(t)=t^{-m} +\O(t^{-m+1})$, so that $\rho_k$ is the unique linear combination of
$\tilde{\rho}_2,\ldots,\tilde{\rho}_k$ with expansion~\eqref{rhokexpansion}
around $\infty$.
We may compute $\rho_k(t)$ as follows:
\begin{enumerate}
  \item Set $\rho_k(t) \colonequals \tilde{\rho}_k(t)$.
  \item For all $m$ from $k-1$ down to 2
    set $$\rho_k(t) \colonequals \rho_k(t) - a_m\tilde{\rho}_m(t),$$
      where $a_m$ is the coefficient of $t^{-m}$ in $\rho_k(t)$.
\end{enumerate}
Since we need a linear number of steps in $k$, and we need to apply this
method to $k=3p^n$ and $k=p^n$, the algorithm gets fairly slow even for
moderate values of $p$ and $n$. 
In practice, 
  we only need to compute $x^i(t)$ for $i=0,\ldots,4$, and we need  $(-y(t))^j$ for $j=1,\ldots, M$, where
    $M=\lfloor k/5\rfloor$ and $j$ is odd, since for odd $m$, there are no 
    even order terms in the expansions of $\tilde{\rho}_m$. Rather than
    computing these anew for every $m$, we first compute
      $(-y(t))^M$ and subsequently multiply by $y(t)^{-2}$. While computing
      $\rho_{3p^n}$, we also compute    $\rho_{p^n}$ along the way.
  To keep the computations feasible, we truncate all Laurent series whenever possible so that the final
    result is correct up to $O(t^5)$, and we work with integers modulo
    $p^n$.

\section{Numerical examples}
\label{sec:Num_Exa}

In this final section, we illustrate our algorithms with numerical
examples. Below, the base field is $\Q$. Moreover, we take $\chi$ to be the
cyclotomic character as in~\S\ref{subsec:QC}.
Let $w$ denote the hyperelliptic involution.

\begin{eg}(Genus~1) 
Let $C/\Q$ be an elliptic curve with split multiplicative reduction at
  a prime number $p$. The canonical height pairing of Mazur--Tate 
  is the same as the Coleman--Gross height pairing with respect to the unit
  root subspace; see, for instance,
  \cite[Remark~3.20]{BKM23}. Here we will confirm this equality numerically
  in an example.

The canonical Mazur--Tate height of a point $P\in
  C(\Q)$ is given in terms of the canonical $p$-adic sigma function
  $\sigma_p$ of \cite{mazur1991p} by 
$2\log_p\left(\frac{e(P)}{\sigma_p(t(P))}\right)$,
with notation as in \cite[\S4.2]{stein2013algorithms}\footnote{There,
  $\hat{h}_p$ is the Schneider height \cite{Sch82}, which is not of interest for this
  example.}, and can be easily computed using {\tt SageMath}. As a concrete example, consider the elliptic curve $C/\Q$ given by 
\begin{equation}\label{eq:EllCur}
y^2 = x^3-1351755x+555015942
\end{equation}
with $P\coloneqq \left(\frac{330483}{361},\frac{63148032}{6859}\right)\in C(\Q)$.
This curve has split multiplicative reduction at $p = 43$. Using {\tt
  SageMath}, we compute that the canonical Mazur--Tate height of $P$ is
\begin{equation}
\label{MTheightEx631}
19 \cdot 43 + 7 \cdot 43^{2} + 8 \cdot 43^{3} + 2 \cdot 43^{4} + 28 \cdot 43^{5} + O(43^{6}).
\end{equation} 
For the two points $Q = (2523,114912)$ and $R = (219,16416)$ on $C$,
let 
$D_Q = (Q) - (-Q)$ and $D_R = (R) - (-R)$; we then have $P = [D_Q] = [D_R]$. We now compute the Coleman--Gross height $h(D_Q,D_R)$ with respect to the unit root subspace and compare the result to~\eqref{MTheightEx631}.

Away from $43$, using the Magma implementation of the algorithm developed
  in \cite{van2020explicit}, we have $\sum_{q\neq 43} h_q(D_Q,D_R) = 9\cdot
  \log_{43}(2)$. For the component at $p = 43$, we first note that the unit
  root subspace is generated by $\alpha[\omega_0] + [\omega_1]$ where
\[\alpha = 17 + 37 \cdot 43 + 20 \cdot 43^{2} + 11 \cdot 43^{3} + 38 \cdot
  43^{4} + 6 \cdot 43^{5} + O(43^{6})\,;\]
  this was computed in {\tt SageMath} using 
  \cite[\S4.2]{stein2013algorithms}.
We then have
\[h_{p}(D_Q,D_R) = \Vint_{-R}^{R} \omega_{D_Q} = \Vint_{-R}^{R} \frac{y(Q)}{x-x(Q)}\frac{dx}{y} + (c_1\alpha - c_0)\Vint_{-R}^{R} \omega_0\]
where $c_0 = \lVint_{Q}^{-Q} \omega_1$ and $c_1 = \lVint_{-Q}^{Q} \omega_0$.
Using the techniques described in \cite{KatzKaya,Kaya}, we compute
\begin{align*}
\Vint_{-R}^{R} \frac{y(Q)}{x-x(Q)}\frac{dx}{y} &= 29 \cdot 43 + 29 \cdot 43^{2} + 18 \cdot 43^{3} + 29 \cdot 43^{4} + 3 \cdot 43^{5} + O(43^{6}),\\
\Vint_{-R}^{R} \omega_0 &= 12 \cdot 43^{2} + 43^{3} + 18 \cdot 43^{4} + 40
  \cdot 43^{5} + O(43^{6}) =
\Vint_{-Q}^{Q} \omega_0\\
\Vint_{Q}^{-Q} \omega_1 &= 18 + 31 \cdot 43 + 8 \cdot 43^{2} + 16 \cdot 43^{3} + 17 \cdot 43^{4} + 8 \cdot 43^{5} + O(43^{6}),
\end{align*}
which give
\[h_{p}(D_Q,D_R) = 29 \cdot 43 + 28 \cdot 43^{2} + 10 \cdot 43^{3} + 39 \cdot 43^{4} + 7 \cdot 43^{5} + O(43^{6}).\]

Putting all of this together, we get
\[
h(D_Q,D_R) = \sum_{q} h_q(D_Q,D_R) = 19 \cdot 43 + 7 \cdot 43^{2} + 8 \cdot 43^{3} + 2 \cdot 43^{4} + 28 \cdot 43^{5} + O(43^{6})
\]
which is the same result as in \eqref{MTheightEx631}.

\begin{rem}
The curve in this example may appear very special since the coefficients of
  its defining polynomial are quite large, but it is not. In order to
  compare the canonical Mazur--Tate height with the
  Coleman--Gross height with respect to the unit
  root subspace, we needed an elliptic curve $C/\Q$ with split
  multiplicative reduction at a prime $p>2$ that satisfies an additional constraint: existence of two distinct points $Q,R\in C(\Q)$ such that
\[D_Q = (Q)-(-Q) = (R)-(-R) = D_R.\] 
In this way, we could apply Algorithm~\ref{alg:CompCG} in order to compute
  the local height $h_{p}(D_Q,D_R)$. A quick search on the LMFDB database
  revealed
  \cite[\href{https://www.lmfdb.org/EllipticCurve/Q/6622/i/3}{6622.i3}]{lmfdb},
  which is the curve in \eqref{eq:EllCur}. Because the mentioned additional constraint is fairly mild, similar computations can be performed for curves given by polynomials with reasonably small coefficients. 
\end{rem}

\end{eg}

\begin{eg}(Genus~2)\label{g2ex1}
Consider the hyperelliptic curve $C/\Q$ given by 
\[y^2 = x^5+5x^4-168x^3+1584x^2-10368x+20736.\]
  Here $p=5$ is a prime of bad semistable reduction for $C$, and the stable reduction is a curve of genus~$0$ with two ordinary double points:
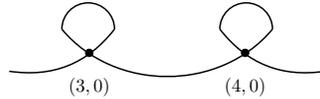
\[\scalebox{0.7}{
\begin{tikzpicture}
\draw[thick] (0,0) to[bend right] (2,0.8);
\draw[thick] (2,0.8) arc (0:180:0.5cm);
\draw (1,0.8)[thick] .. controls (2,-0.4) and (4,-0.4) .. (5,0.8);
\draw[thick] (5,0.8) arc (0:180:0.5cm);
\draw[thick] (4,0.8) to[bend right] (6,0);
\filldraw (1.52,0.35) circle (2pt);
\filldraw (4.48,0.35) circle (2pt);
\filldraw (1.52,-0.3) node{$(3,0)$};
\filldraw (4.48,-0.3) node{$(4,0)$};
\end{tikzpicture}}\]
Consider the points $P = (-12,720)$, $Q = (-8,528)$, $R = (0,-144)$, $S = (12,432)$, $T = (36,7920)$ on $C$ and set
\[
D_1 = (Q) - (w(Q)), \ \
D_2 = (R) - (P), \ \
D_3 = (S) - (w(P)), \ \
D_4 = (T) - (w(T)).
\]
  One can easily check that we have $[D_1] =
  4[D_2]$ and $[D_4] = 6[D_3]$, which implies
\begin{equation}
\label{EqualityOfHeights}
6h(D_1,D_3) = h(D_1,D_4) = h(D_4,D_1) = 4h(D_4,D_2)\,,
\end{equation}
provided that the chosen complementary subspace for the prime $5$ is isotropic with respect to the cup product pairing. In this example, we verify this up to a certain precision.

Away from $5$, using again the algorithm developed in~\cite{van2020explicit}; we have
\[\sum_{q\neq 5} h_q(D_1,D_3) = 0,\ \ \ \ \ \ \ 
\sum_{q\neq 5} h_q(D_4,D_2) = -2\cdot \log_p(2) + \log_p(3).\]
For the local contributions at $p = 5$, let $W^C_p$ be the canonical subspace. Using the algorithm described in \S\ref{sec:CompBlakSubs}, we find that this space is generated by the differentials
\[
[\eta_0] = k_0[\omega_0] + k_1[\omega_1] + k_2[\omega_2]\ \ \ \text{and} \ \ \ 
[\eta_1] = l_0[\omega_0] + l_1[\omega_1] + l_2[\omega_2] + l_3[\omega_3],\]
where 
\begin{align*}
k_0 &= 2 + 3 \cdot 5 + 4 \cdot 5^{2} + 2 \cdot 5^{3} + 5^{4} + 3 \cdot 5^{6} + O(5^{7}), \\
k_1 &= 3 + 4 \cdot 5 + 5^{3} + 3 \cdot 5^{4} + 2 \cdot 5^{5} + 4 \cdot 5^{6} + O(5^{7}), &k_2 &= 1, \\
l_0 &= 2 + 3 \cdot 5 + 2 \cdot 5^{2} + 4 \cdot 5^{3} + 5^{4} + 2 \cdot 5^{5} + 5^{6} + O(5^{7}), &l_2 &= 15, \\
l_1 &= 4 + 2 \cdot 5 + 2 \cdot 5^{2} + 2 \cdot 5^{3} + 2 \cdot 5^{4} + 3 \cdot 5^{5} + O(5^{7}), &l_3 &= 3.
\end{align*}

Let us compute $h_p(D_1,D_3)$. The form $\omega \coloneqq \frac{y(Q)}{x-x(Q)}\frac{dx}{y}$
satisfies $\Res(\omega) = D_1$ and
\begin{equation}
\label{Ex:PsiOmegaInd'se's}
\Psi(\omega) = d_0[\omega_0] + d_1[\omega_1] + e_0[\eta_0] + e_1[\eta_1]
\end{equation}
for some $d_0,d_1,e_0,e_1$. We then have
\[h_p(D_1,D_3) = \Vint_{w(P)}^{S} \omega_{D_1} = \Vint_{w(P)}^{S} \frac{y(Q)}{x-x(Q)}\frac{dx}{y} - d_0\Vint_{w(P)}^{S} \omega_0 - d_1\Vint_{w(P)}^{S} \omega_1.\]
Using~\cite{KatzKaya,Kaya}, we compute
\begin{align*}
\Vint_{w(P)}^{S} \frac{y(Q)}{x-x(Q)}\frac{dx}{y} &= 2 \cdot 5 + 5^{2} + 4 \cdot 5^{3} + 3 \cdot 5^{4} + 4 \cdot 5^{5} + 2 \cdot 5^{6} + O(5^{7}), \\
\Vint_{w(P)}^{S} \omega_0 &= 3 \cdot 5 + 3 \cdot 5^{2} + 5^{4} + 4 \cdot 5^{5} + O(5^{7}), \\
\Vint_{w(P)}^{S} \omega_1 &= 3 \cdot 5 + 4 \cdot 5^{2} + 3 \cdot 5^{3} + 4 \cdot 5^{4} + 2 \cdot 5^{6} + O(5^{7}).
\end{align*}
In order to determine $d_0$ and $d_1$, write
\begin{equation}
\label{Ex:PsiOmegaInC's}
\Psi(\omega) = c_0[\omega_0] + c_1[\omega_1] + c_2[\omega_2] + c_3[\omega_3].
\end{equation}
Comparing \eqref{Ex:PsiOmegaInd'se's} and \eqref{Ex:PsiOmegaInC's}, we see that
\begin{align*}
e_1 &= c_3/l_3,\qquad
&& e_0 = (c_2-e_1\cdot l_2)/k_2,\\
d_1 &= c_1 - e_0\cdot k_1 - e_1\cdot l_1, \qquad
&& d_0 = c_0 - e_0\cdot k_0 - e_1\cdot l_0.
\end{align*}
The cup product matrix $N$ with respect to the basis $\B = \{[\omega_0],[\omega_1],[\omega_2],[\omega_3]\}$ is
\[\left(\begin{array}{rrrr}
0 & 0 & 0 & 1/3 \\
0 & 0 & 1 & -10/3 \\
0 & -1 & 0 & -56 \\
-1/3 & 10/3 & 56 & 0
\end{array}\right)\]
and the global symbols are
\[\begin{pmatrix}
\angles{\omega,\omega_0} \\
\angles{\omega,\omega_1} \\
\angles{\omega,\omega_2} \\
\angles{\omega,\omega_3} 
\end{pmatrix} = \begin{pmatrix}
2 \cdot 5 + 5^{2} + 4 \cdot 5^{3} + 4 \cdot 5^{4} + 3 \cdot 5^{5} + 4 \cdot 5^{6} + O(5^{7})\\
4 \cdot 5 + 2 \cdot 5^{2} + 4 \cdot 5^{3} + 4 \cdot 5^{4} + 2 \cdot 5^{5} + O(5^{7})\\
2 + 3 \cdot 5 + 2 \cdot 5^{3} + 4 \cdot 5^{4} + 4 \cdot 5^{5} + O(5^{7})\\
2 \cdot 5^{2} + 2 \cdot 5^{3} + 3 \cdot 5^{4} + 2 \cdot 5^{5} + 4 \cdot 5^{6} + O(5^{7})
\end{pmatrix}.\]
We then get that
\[\begin{pmatrix}
c_0 \\
c_1 \\
c_2 \\
c_3 
\end{pmatrix} = \begin{pmatrix}
2 \cdot 5 + 2 \cdot 5^{2} + 2 \cdot 5^{3} + 4 \cdot 5^{4} + 3 \cdot 5^{6} + O(5^{7}) \\
2 + 4 \cdot 5 + 5^{3} + 2 \cdot 5^{5} + 5^{6} + O(5^{7}) \\
5 + 3 \cdot 5^{2} + 2 \cdot 5^{3} + 5^{4} + 2 \cdot 5^{5} + 5^{6} + O(5^{7}) \\
4 \cdot 5 + 2 \cdot 5^{3} + 3 \cdot 5^{5} + O(5^{7})
\end{pmatrix}\]
which implies that
\[\begin{pmatrix}
e_1 \\
e_0 \\
d_1 \\
d_0 
\end{pmatrix} = \begin{pmatrix}
3 \cdot 5 + 3 \cdot 5^{2} + 5^{5} + O(5^{7}) \\
5 + 4 \cdot 5^{2} + 5^{3} + 4 \cdot 5^{4} + 5^{5} + 3 \cdot 5^{6} + O(5^{7}) \\
2 + 4 \cdot 5 + 3 \cdot 5^{2} + 2 \cdot 5^{3} + 3 \cdot 5^{4} + 2 \cdot 5^{5} + 5^{6} + O(5^{7}) \\
4 \cdot 5 + 4 \cdot 5^{2} + 3 \cdot 5^{3} + 4 \cdot 5^{4} + 4 \cdot 5^{6} + O(5^{7})
\end{pmatrix}.\]
Combining all of this, we obtain
\[h_p(D_1,D_3) =  5 + 3 \cdot 5^{2} + 2 \cdot 5^{3} + 2 \cdot 5^{4} + 3 \cdot 5^{5} + 4 \cdot 5^{6} + O(5^{7}).\]
A similar computation, which we omit, gives that
\[h_p(D_4,D_2) = 4 \cdot 5 + 2 \cdot 5^{2} + 2 \cdot 5^{3} + 4 \cdot 5^{4} + 4 \cdot 5^{5} + 4 \cdot 5^{6} + O(5^{7}).\]
We recomputed these local heights using
the code~\cite{BiaCod}, and recovered the same results,
confirming~\cite[Corollary~5.32]{Bia23} in this example.

Finally, we see that
\begin{align*}
h(D_1,D_3) &= \sum_{q} h_q(D_1,D_3) = 5 + 3 \cdot 5^{2} + 2 \cdot 5^{3} + 2 \cdot 5^{4} + 3 \cdot 5^{5} + 4 \cdot 5^{6} + O(5^{7}),\\
h(D_4,D_2) &= \sum_{q} h_q(D_4,D_2) = 4 \cdot 5 + 4 \cdot 5^{2} + 3 \cdot 5^{3} + 3 \cdot 5^{4} + 2 \cdot 5^{5} + 4 \cdot 5^{6} + O(5^{7}),
\end{align*}
and hence get the following result, as predicted by~\eqref{EqualityOfHeights}
\[
6h(D_1,D_3) = 5 + 4 \cdot 5^{2} + 5^{5} + 3 \cdot 5^{6} + O(5^{7}) = 4h(D_4,D_2).
\]
\end{eg}

\begin{eg}(Genus~2) Consider the hyperelliptic curve $C/\Q$ given by 
\[y^2 =  x^5 - 4x^4 - 48x^3 + 64x^2 + 512x + 256\]
which is obtained from the curve
\cite[\href{https://www.lmfdb.org/Genus2Curve/Q/125237/a/125237/1}{125237.a.125237.1}]{lmfdb}
  by applying the change of variables $x\mapsto x/4$ and $y\mapsto y/16$.
  The prime $p=7$ is a prime of bad reduction for $C$ and the corresponding
  (stable) reduction is an elliptic curve with an ordinary double point:
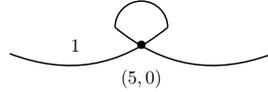
\[\scalebox{0.7}{
\begin{tikzpicture}
\draw[thick] (-3,0) to[bend right] (0,0.5);
\draw[thick] (-1,0.5) to[bend right] (2,0);
\draw[thick] (0,0.5) arc (0:180:0.5cm);
\filldraw (-1.75,0.15) node{$1$};
\filldraw (-0.5,0.17) circle (2pt);
\filldraw (-0.5,-0.48) node{$(5,0)$};
\end{tikzpicture}}\]
Consider the points $P = (-4,16)$, $Q = (-3,-5)$, $R = (0,16)$, $S = (4,16)$, $T = (8,16)$, $U = (36,7184)$ on $C$, and set 
\begin{align*}
D_1 &= (P)-(w(P)), \ \ \ & D_2 &= (R)-(w(R)), & D_3 &= (Q)-(S),\\    
D_4 &= (T)-(U), \ \ \    & D_5 &= (w(T))-(w(U)).
\end{align*}

The canonical subspace $W^C_p$ is generated by 
\[
[\eta_0] = k_0[\omega_0] + k_1[\omega_1] + k_2[\omega_2]\ \ \ \text{and} \ \ \ 
[\eta_1] = l_0[\omega_0] + l_1[\omega_1] + l_2[\omega_2] + l_3[\omega_3],\]
where 
\begin{align*}
k_0 &= 7 + 6 \cdot 7^{2} + 2 \cdot 7^{3} + 5 \cdot 7^{4} + 2 \cdot 7^{5} + O(7^{6}), \\
k_1 &= 2 + 6 \cdot 7 + 7^{2} + 4 \cdot 7^{3} + 3 \cdot 7^{4} + 7^{5} + O(7^{6}), &k_2 &= 1, \\
l_0 &= 2 + 7 + 6 \cdot 7^{2} + 4 \cdot 7^{4} + 7^{5} + O(7^{6}), &l_2 &= -12, \\
l_1 &= 4 \cdot 7 + 4 \cdot 7^{2} + 6 \cdot 7^{3} + 4 \cdot 7^{4} + 3 \cdot 7^{5} + O(7^{6}), &l_3 &= 3.
\end{align*}
Following the same steps as in Example~\ref{g2ex1}, we find that
\begin{align*}
h(D_1,D_2) &= \boldsymbol{0} + 7 + 6 \cdot 7^{2} + 2 \cdot 7^{3} + 3 \cdot 7^{4} + 3 \cdot 7^{5} + 6 \cdot 7^{6} + O(7^{7}), \\
           &= 7 + 6 \cdot 7^{2} + 2 \cdot 7^{3} + 3 \cdot 7^{4} + 3 \cdot 7^{5} + 6 \cdot 7^{6} + O(7^{7}), \\
h(D_3,D_4) &= \boldsymbol{-3\cdot\log_p(2) + \log_p(13)} + 3 \cdot 7^{3} + 2 \cdot 7^{5} + 6 \cdot 7^{7} + O(7^{8}),\\
           &= 4 \cdot 7 + 7^2 + 3 \cdot 7^3 + 7^4 + 2 \cdot 7^5 + 4 \cdot 7^6 + 6 \cdot 7^7 + O(7^8), \\
h(D_3,D_5) &= \boldsymbol{\log_p(3) - \log_p(11)} + 2 \cdot 7^{2} + 3 \cdot 7^{3} + 5 \cdot 7^{4} + 3 \cdot 7^{5} + 2 \cdot 7^{6} + O(7^{7}),\\
           &= 3 \cdot 7 + 5 \cdot 7^2 + 3 \cdot 7^3 + 5 \cdot 7^4 + 4 \cdot
           7^5 + 2 \cdot 7^6 + O(7^7)\,, 
\end{align*}
where the bold-faced terms are the local contributions away from $p = 7$.
Note that the computation of the local heights at $7$ is somewhat more
involved than in Example~\ref{g2ex1}, because the computation of the relevant Vologodsky integrals
  requires the computation of integrals of differentials of the third kind
  on an elliptic curve with good reduction (reducing to the elliptic curve
  in the stable reduction at $7$), whereas in Example~\ref{g2ex1}, the
  actual Coleman integrals were all computed on $\P^1$.
  
Again, we used~\cite{BiaCod} to show that these results are consistent with~\cite[Corollary~5.32]{Bia23}.

\end{eg}

\begin{eg}\label{QCex}(Quadratic Chabauty) 
According to
\cite[\href{https://www.lmfdb.org/Genus2Curve/Q/85280/d/682240/1}{85280.d.682240.1}]{lmfdb},
the genus 2 curve
  \begin{equation}\label{Ceqn}
C\colon y^2 = x^5 + x^3 - 2x + 1
\end{equation}
has good reduction away from $2,3,5,41$ and its Jacobian $J$ 
satisfies
$$
  J(\Q) = \langle P_1, P_2\rangle \cong \Z\times
  \Z\,;\quad P_1= [(0,1)-(1,1)],\; P_2= [(1,-1)-\infty].
$$
  We now use quadratic Chabauty at the bad prime $p=5$ as outlined
  in~\S\ref{subsec:QC} to show that the set
  of integral points on $C$ consists precisely of the points 
  \begin{equation}\label{intpts}
    (0,\pm 1), (1,\pm1), (-1,\pm 1)\,.
  \end{equation}
If $\chi_v$ is unramified,
  then~\cite[Proposition~3.3]{balakrishnan2016quadratic} implies that the function that
  maps an integral point $P \in C_v(\Z_v)$, to
  $h_v(P-\infty, P-\infty)$, normalised as in
  Remark~\ref{R:common_compute}, is constant on preimages of irreducible
  components of any proper
  regular model of $C_v$ over $\Z_v$, and is identically~0 when there is
  such a model with irreducible special fibre. Hence the same is true for the
  function $\nu_v$ defined in~\eqref{nuv} (see also~\cite[Lemma~7.2]{Bia23}
  and~\cite[Theorem~7.1]{Bia23}, which
  summarises various results from~\cite{muller2016canonical}).
  The equation~\eqref{Ceqn} defines a
  proper regular model for all $v\ne 2$, which implies that the sets
in Theorem~\ref{T:qc} satisfy $\Gamma_v=\{0\}$ for all $v\ne 2$.
  At~2, the minimal proper regular model has four components, all of
  genus~0. The value of $\nu_2$ is always~$0$ for the component
  containing~$0$, since this holds for the local height
  $h_v(P-\infty,P-\infty)$
  by~\cite[Proposition~3.3]{balakrishnan2016quadratic}. The $\Z_2$-integral points
  $(1,1), (-1,1), (4/9, 339/4)$ map to the other three components, respectively,
  and we compute that $\nu_2$ takes the values
  $-{2}\log_p(2)$ and $-\frac{8}{3}\log_p(2)$ there. 
  Thus the set $\Gamma$ in Theorem~\ref{T:qc} is
\begin{equation*}
\Gamma=\Gamma_2 = \left\{0, {2}\log_p(2), \frac{8}{3}\log_p(2)\right\}\,.
\end{equation*}
We work with the bad prime $p=5$, and we choose the subspace $W_5$
  corresponding to the matrix $c=0$. Since $x^5 + x^3 - 2x + 1$ reduces to 
  $(x+2)^2(x^3 + x^2 + 3x + 4)$,~\cite[Figure~1]{bruin_stoll_2010}
  and~\cite[Theorem~7.1]{Bia23} imply
  that $mP$ lies in the kernel of reduction of  $J(\Q_5)$ for any $P\in
  J(\Q_5)$, where $m=(p-1)\#E(\F_5) = 4\cdot 9=36$  and $E/\F_5$ is the
  elliptic curve defined by $y^2 = x^3 + x^2 + 3x + 4$.  In this example,
  the kernel of reduction equals the subgroup $H_5$
  by~\cite[Section~3]{Bia23}. Hence we may always use the factor~$36$ to
  compute the abelian logarithm $\cL$ on $J(\Q_5)$ and the N\'eron function
  $\lambda_5$.   
  We use this to express the global height $h$ as a linear combination of
  the products $\cL_1^2,\cL_1\cL_2,\cL_2^2$ by evaluating these in the
  points $P_1, P_3 = [(0,1)-(1,-1)]\in
  J(\Q)$, which generate a finite index subgroup of
  $J(\Q)$ and which have the property that $P_1,P_3,P_1+P_3\in
  J(\Q)\setminus\Supp(\Theta)$. The latter condition is not necessary if we
  compute the heights and the $\cL_i$ using Vologodsky integration as in
  Section~\ref{sec:Comp_CG_height} and in~\cite{KatzKaya,Kaya}, but is required
  to apply the results of~\cite{Bia23} discussed in Section~\ref{sec:4},
  which is what we used for this example. We find that
  \begin{equation}\label{hlogs}
    h = 
    a_1\cL_1^2-a_2\cL_1\cL_2-a_3\cL_2^2\,,
  \end{equation}
where 
\begin{align*}
  a_1 &=4\cdot 5^{-1} + 1 + 4\cdot 5^3 + 5^4 + 5^6 + O(5^7), \\
  a_2 &=4\cdot 5^{-1} + 3 + 5 + 4\cdot 5^6 + O(5^7), \\ 
  a_3 &=1 + 4\cdot 5 + 5^3 + 4\cdot 5^4 + 3\cdot 5^5 + 4\cdot 5^6 +
  O(5^7)\,.
\end{align*}
Hence the quadratic Chabauty function $\rho$ in Theorem~\ref{T:qc} is
  \begin{equation}\label{rho}
    \rho\colonequals \nu_5 -
    a_1\cL_1^2-a_2\cL_1\cL_2-a_3\cL_2^2\,,
  \end{equation}
  with $a_i$'s as above and $\nu_5$ defined in~\eqref{nuv}.

  In order to solve for the set of all $P\in C(\Z_5)$ such that $\rho(P)\in
  \Gamma$, we need to expand $\rho$ in residue discs. Since no $\Q_5$-point
  lies in the residue disc reducing to the singular point $(3,0)\in
  C(\F_5)$, we only need to consider the discs reducing to $(0, \pm1), (1,
  \pm 1), (-1,\pm 1)$. These all contain at least one integral point
  $(x_0,y_0)$, which
  we may use to parametrize them using the local parameter $z=x-x_0$. Also note that by symmetry, we only need
  to find the roots of $\rho$ in these disc up to the hyperelliptic
  involution. 
  In order to expand the functions $\nu_5$ and $\cL_i$ in the local
  parameter $z$, we need to work over the Laurent series
  ring $\Q[[z]]$ (see for instance~\eqref{lambdavdef}, noting that the
  coefficients of the expansion of the coordinates of a parametric point
  $P(z)$ on $C$ have rational coefficients). In particular, we
  need to perform scalar multiplication by~36 on $J(\Q[[z]])$.
  Rather than using Cantor's algorithm, which does not behave well over
  inexact fields, we instead perform scalar multiplication on the Kummer
  surface $K\colonequals J/\langle -1\rangle$. This is a
  quartic hypersurface that remembers much of the arithmetic on $J$. In
  particular, let $\kappa\colon J\to K\subset \P^3$ be the map
  in~\cite{FlynnSmart}, then, for any $P\in J$ and $n\in \Z$, we can
  compute $\kappa(nP)$ from $\kappa(P)$ by repeatedly evaluating certain
  quartic forms and biquadratic forms as
  in~\cite{FlynnSmart} -- without loss of precision. From $\kappa(nP)$, we can then find the unordered
  pair $\{\pm nP\}$. Since the functions appearing in $\rho$ are even (in
  Grant's formal group parameters $T_1,T_2$), this
  ambiguity is not an issue for us.

  After solving $\rho\in \Gamma$, we find, in addition to the known integral
  points~\eqref{intpts}, approximations of four ``extra'' $p$-adic points on $C$:
  \begin{align*}
    &{\scriptstyle(1 + 2\cdot 5 + 2\cdot 5^3 + 2\cdot 5^4 + 2\cdot 5^5 +
    O(5^6)\, ,\; 1
    + 5 + 4\cdot 5^2 + 3\cdot 5^3
    + 4\cdot 5^4 + 5^5 +  O(5^6)),}\\
    &{\scriptstyle(4\cdot 5 + 2\cdot 5^2 + 3\cdot 5^3 + O(5^6) \,,\; 1 + 5 + 4\cdot
    5^2 + 5^3 + 5^4 + 4\cdot 5^5
    + O(5^6)),}\\
    &{\scriptstyle(4 + 2\cdot 5 + 2\cdot 5^2 + 4\cdot 5^3 + 3\cdot 5^4 + 2\cdot 5^5 + O(5^6)
    \,,\; 4 + 5^2 + 5^3 + 3\cdot 5^4
    + 3\cdot 5^5 + O(5^6)),}\\
    &{\scriptstyle(4 + 5 + 4\cdot 5^2 + 4\cdot 5^3 + 4\cdot 5^4 + O(5^6)
    \,,\; 4 + 3\cdot 5 +
    2\cdot 5^3 + 2\cdot 5^5 + O(5^6))}
  \end{align*}
  It remains to show that these points are not $\Z$-integral.
  To this end we use the Mordell--Weil sieve as
  in~\cite{balakrishnan2017computing}: Assuming that an extra
  point $P$ is actually in $C(\Q)$, there are
  coefficients $a_1,a_2\in\Z$ satisfying 
  $ \iota(P) = a_1P_1+a_2P_2$.
  We use the
  linearity of the abelian logarithms $\cL_1,\cL_2$, again relying on
  Kummer arithmetic to avoid loss of precision in arithmetic in $J(\Q_p)$, to approximate these
  putative coefficients modulo $5^6$, which results in a coset of
  $5^6J(\Q)$ in
  $J(\Q)$ that would have to contain $\iota(P)$.
  We show that the image of
  this coset in $J(\F_{311})/5^6J(\F_{311})$ does not intersect the image of
  $C(\F_{311})$ for any extra point, which finally proves that $C(\Z)$
  consists precisely of the points in~\eqref{intpts}. 

\end{eg}

\subsection*{Competing interest statement} The authors have no competing interests to declare that are relevant to the content of this article.

\subsection*{Data availability statement}

The code to reproduce our computations were written using \texttt{SageMath} \cite{SageMath} and \texttt{Magma} \cite{Magma}, and are available in the public GitHub repository
\[
\text{\url{https://github.com/KayaEnis/padicHeights}}.
\]

\bibliographystyle{amsalpha}
\bibliography{master}

\end{document}